\documentclass[12pt]{amsart}

\usepackage{amsmath}
\usepackage{amscd}
\usepackage{eucal}
\usepackage{amssymb}
\usepackage{epic}
\usepackage{graphics}
\usepackage{color}
\usepackage{mathrsfs}
\usepackage{multirow}
\usepackage{subfigure}
\usepackage{graphicx}
\usepackage{mathtools}
\usepackage{float}
\usepackage[table]{xcolor}
\usepackage{anyfontsize}

\usepackage{caption,tikz}
\usepackage{bbm}
\usepackage{rotating}
\usetikzlibrary{arrows}

\numberwithin{equation}{section}
\numberwithin{table}{section}

\theoremstyle{plain}
\newtheorem{theorem}{Theorem}[section]
\newtheorem{lemma}{Lemma}[section]

\theoremstyle{remark}

\newtheorem*{remarks}{Remarks}

\newcommand{\ZZZ}{{\mathbb{Z}}}
\newcommand{\ZZ}{{\mathbf{Z}}}

\newcommand{\0}{{\textcolor{white}{000}}}

\DeclareMathOperator{\coker}{coker}
\DeclareMathOperator{\im}{Im}

\DeclareMathOperator{\K}{\mathcal{K}}

\DeclareMathOperator{\srg}{srg}

\title[Invariants of the Rook's Graph]{The Smith and Critical Groups of the Square Rook's Graph and its Complement}
\author[Ducey]{Joshua E. Ducey} 
\author[Gerhard]{Jonathan Gerhard}
\author[Watson]{Noah Watson}

\address{Dept.\ of Mathematics and Statistics, James Madison University, Harrisonburg, VA 22807}
\email{duceyje@jmu.edu}
\email{gerha2jm@dukes.jmu.edu}
\email{watsonnj@dukes.jmu.edu}


\begin{document}
\begin{abstract}
Let $R_{n}$ denote the graph with vertex set consisting of the squares of an $n \times n$ grid, with two squares of the grid adjacent when they lie in the same row or column.  This is the square rook's graph, and can also be thought of as the Cartesian product of two complete graphs of order $n$, or the line graph of the complete bipartite graph $K_{n,n}$.  In this paper we compute the Smith group and critical group of the graph $R_{n}$ and its complement.  This is equivalent to determining the Smith normal form of both the adjacency and Laplacian matrix of each of these graphs.  In doing so we verify a 1986 conjecture of Rushanan.

\end{abstract}
\maketitle
\section{Introduction}
This paper concerns integer invariants of the square rook's graph, denoted $R_{n}$.  Such a graph is formed by taking as its vertex set the squares of an $n \times n$ grid, and defining two squares to be adjacent when they lie in the same row or the same column.  That is, if one decided to play chess on such a grid, two squares are adjacent when a rook can move from one to the other.  It is easy to see that $R_{n}$ is isomorphic to the Cartesian product of two complete graphs of order $n$, and also to the line graph of the complete bipartite graph $K_{n,n}$.  

To avoid trivialities, we let $n \geq 3$ throughout.  The graph $R_{n}$ is a \textit{strongly regular graph} (srg) with parameters 
\begin{align*}
v &= n^2 \\
k &= 2(n-1)\\
\lambda &= n-2\\
\mu &= 2
\end{align*}
and is in fact determined up to isomorphism by these parameters, except in the case $n=4$.  See \cite{b-h, moon, shrikhande} for more information and many interesting properties.

When considering a graph $\Gamma$, a standard method of study is to encode its information into a matrix and then work to understand algebraic properties of the matrix.  This can sometimes yield useful information about the graph, and is one of the basic ideas of spectral graph theory.  Another graph invariant, somewhat more subtle than the spectrum, is an abelian group that can be read off from the Smith normal form of such a matrix.  We defer formal definitions until the next section.  

From the adjacency matrix we get the \textit{Smith group} $S(\Gamma)$ and from the Laplacian matrix we get the \textit{critical group} $\K(\Gamma)$.  Our main theorem will be the determination of these groups for the rook's graph $R_{n}$ and its complement $R_{n}^{c}$.   We will write $\ZZ_{r}$ for the group $\ZZZ / r \ZZZ$.

\begin{theorem} \label{thm:main}
Let $R_{n}$ denote the square rook's graph.  The critical group and Smith group of $R_{n}$ and its complement $R_{n}^{c}$ are given by the following isomorphisms:
\begin{align}
\K(R_{n}) &\cong \left(\ZZ_{2n}\right)^{(n-2)^2 + 1} \oplus \left(\ZZ_{2n^{2}}\right)^{2(n-2)} \label{iso:KR}\\
S(R_{n}) &\cong \left(\ZZ_{2}\right)^{(n-2)^{2}} \oplus \left(\ZZ_{2(n-2)}\right)^{2n-3} \oplus \ZZ_{2(n-1)(n-2)} \label{iso:SR}\\
\K(R_{n}^{c}) &\cong \left(\ZZ_{n(n-2)}\right)^{(n-2)^{2}-1} \oplus \left(\ZZ_{n(n-1)(n-2)}\right)^{2} \oplus \left(\ZZ_{n^{2}(n-1)(n-2)}\right)^{2(n-2)} \label{iso:KRc}\\
S(R_{n}^{c}) &\cong \left(\ZZ_{(n-1)}\right)^{2(n-1)} \oplus \ZZ_{(n-1)^{2}} \label{iso:SRc}.
\end{align}
\end{theorem}

\begin{remarks} \hfil 
\begin{enumerate}
\item In \cite[Example SNF2]{b-ve}, the authors state isomorphism \ref{iso:SR}, and more generally give a diagonal form for the matrix $A+(c+2)I$, where $A$ is the adjacency matrix of $R_{n}$.  They then refer the reader to van Eijl's thesis \cite{ve} for the ``somewhat boring proof''.  In \cite{berget} the author computes the critical group of the line graph of the complete bipartite graph $K_{n,m}$ (i.e., the $n \times m$ rook's graph) via integral row and column operations on the Laplacian.  Thus the isomorphisms \ref{iso:KR} and \ref{iso:SR} can be deduced from these works.  We believe the isomorphisms \ref{iso:KRc} and \ref{iso:SRc} are new.
\item  The two isomorphisms above for $S(R_{n})$ and $S(R_{n}^{c})$ appeared as conjectures in \cite[Examples 4-5 and 4-6]{rushanan}.
\end{enumerate}
\end{remarks}

Computation of the Smith normal form of the adjacency or Laplacian matrix is a standard technique to determine the Smith or critical group of a graph.  It is well known that this can be achieved through integral row and column operations, and this approach is popular among many authors.  As mentioned in the first remark above, this is how the rook's graphs' invariants have thus far been approached.  

Our proof is quite different, and handles all of these isomorphisms uniformly.  We will work directly within the particular abelian group to identify a cyclic decomposition for it.  Visualizing the group elements in terms of configurations on the graphs will play an important role, and will nicely illuminate the relationship between the graphs and the groups.  We note that the most natural decomposition is not always the invariant factor decomposition that arises from Smith normal form.  See section \ref{sec:KRC} below, and see \cite{wilson} for a rather spectacular example.  For surveys on Smith groups see \cite{rushanan, sin}.  For a good starting point for critical groups, see \cite{dino2, dino1}.

\section{Preliminaries}
Let $M$ be an $m \times n$ matrix with integer entries.  We may view the matrix as defining a homomorphism of free abelian groups:
\[
M \colon \ZZZ^{n} \to \ZZZ^{m}.
\]
It is the cokernel of this map, $\ZZZ^{m} / \im(M)$, that we are interested in.  This finitely generated abelian group becomes a graph invariant when we take the matrix $M$ to be the adjacency or Laplacian matrix of the graph.

Formally, let $\Gamma$ be a simple graph and order the vertices in some arbitrary but fixed manner.  Let $A = (a_{i,j})$ be a matrix with rows and columns indexed by the vertex set of $\Gamma$.  Set
\[
a_{i,j} = 
	\begin{cases}
	1, & \mbox{if vertex $i$ and vertex $j$ are adjacent}\\
	0, & \mbox{otherwise}.
	\end{cases}
\]
Then $A$ is the \textit{adjacency matrix} of $\Gamma$.  Define a matrix $D = (d_{i,j})$ of the same size as $A$ by 
\[
d_{i,j} = 
	\begin{cases}
	\mbox{the degree of vertex $i$}, & \mbox{if $i = j$}\\
	0, & \mbox{otherwise},
	\end{cases}
\]
and set $L = D - A$.  The matrix $L$ is the \textit{Laplacian matrix} of $\Gamma$.  The adjacency matrix and Laplacian matrix will be our primary focus.  The cokernel of $A$ is known as the Smith group of the graph and is denoted $S(\Gamma)$. The cokernel of $L$ always has free rank equal to the number of connected components of the graph.  The torsion subgroup of the cokernel of $L$ is known as the critical group (or sandpile group, Jacobian group, Picard group) of $\Gamma$ and is denoted $\K(\Gamma)$.  The critical group is especially interesting; for connected graphs, its order is equal to the number of spanning trees of the graph.

One way to compute the cokernel of $M$ is by using the \textit{Smith normal form}.  It was H.J.S. Smith \cite{smith2} who first proved that there exist square, unimodular (i.e. determinant $\pm 1$) integer matrices $P$ and $Q$ so that $PMQ = S$, where the matrix $S=(s_{i,j})$ satisfies:
\begin{enumerate}
\item $s_{i,i}$ divides $s_{i+1,i+1}$ for $1 \leq i < \min \{m,n\}$
\item $s_{i,j} = 0$ for $i \neq j$.
\end{enumerate}
We then have
\[
\coker M \cong \ZZ_{s_{1,1}} \oplus \ZZ_{s_{2,2}} \oplus \cdots
\]
and this is the so-called invariant factor decomposition of the abelian group.  For this reason the integers $s_{i,i}$ are known as the \textit{invariant factors} of $M$; their prime power factors are the \textit{elementary divisors} of $M$.  Any diagonal form of $M$ achieved by unimodular matrices will give a cyclic decomposition of $\coker M$, and serves to identify the group.

We will use the following lemma repeatedly.
\begin{lemma}\label{lem:base}
Let $G$ be a finite abelian group, generated by the elements $x_1, x_2, \ldots, x_k$.  Suppose that there exist integers $r_1, r_2, \ldots , r_k$ so that $|G| = r_1 \cdot r_2  \cdots r_k$ and $|x_{i}|$ divides $r_{i}$, for $1 \leq i \leq k$.  Then
\[
G \cong \ZZ_{r_{1}} \oplus \ZZ_{r_{2}} \oplus \cdots \oplus \ZZ_{r_{k}}.
\]
\end{lemma}
\begin{proof}
Since the order of each $x_{i}$ divides $r_{i}$, there is a homomorphism 
\[
\ZZ_{r_{i}} \to G
\]
that sends $1 \mapsto x_{i}$.  These homomorphisms extend to a unique map
\[
\ZZ_{r_{1}} \oplus \ZZ_{r_{2}} \oplus \cdots \oplus \ZZ_{r_{k}} \to G.
\]
Since the set $\{x_i\}$ generates $G$, this map is onto, and since both groups have the same order, it must be an isomorphism.
\end{proof}

\section{The Rook's Graph and Its Complement}

To prove Theorem \ref{thm:main} we will apply Lemma \ref{lem:base} to each of the four abelian groups.  There are three ingredients needed to apply the lemma:  we need to know the orders of our groups, exhibit a set of elements and show that their orders divide the orders of the cyclic factors as stated in the theorem, and show these elements do indeed generate the group.  We will compute orders of the groups immediately below.  The two subsections following will provide the arguments for the orders of the said generators, as well as the proof that they do indeed form a generating set.

As the rook's graph $R_{n}$ is an $\srg(n^{2}, 2(n-1), n-2, 2)$, we immediately get its adjacency spectrum \cite[Chapter 9]{b-h}:
\[
[-2]^{(n-1)^{2}}, [n-2]^{2n-2}, [2(n-1)]^{1}.
\]
The exponents here indicate multiplicity of the eigenvalue.  Since the adjacency matrix of $R_{n}$ is nonsingular, the product of the eigenvalues equals the product of the invariant factors.  We deduce:
\begin{align*}
|S(R_{n})| &= 2^{(n-1)^{2}} \cdot (n-2)^{2n-2} \cdot 2(n-1) \\
&= 2^{(n-2)^{2}} \cdot (2(n-2))^{2n-3} \cdot 2(n-1)(n-2).
\end{align*}

Since $R_{n}$ is regular of degree $2(n-1)$, the Laplacian spectrum can be immediately deduced from the adjacency spectrum.  We get:
\[
[2n]^{(n-1)^{2}}, [n]^{2n-2}, [0]^{1}.
\]
Kirchhoff's Matrix-Tree Theorem \cite[Prop. 1.3.4]{b-h} implies that the order of the critical group of $R_{n}$ is the product of the nonzero eigenvalues of the Laplacian, divided by the number of vertices.  We get:
\begin{align*}
|\K(R_{n})| &= \frac{1}{n^{2}} \cdot (2n)^{(n-1)^{2}} \cdot n^{2n-2} \\
&= (2n)^{(n-2)^{2}+1} \cdot (2n^{2})^{2(n-2)}.
\end{align*}

The complement graph $R_{n}^{c}$ is also strongly regular, with parameters $\srg(n^{2}, (n-1)^{2}, (n-2)^{2}, (n-1)(n-2))$.  The exact same calculation as above can be repeated to compute the group orders $|S(R_{n}^{c})|$ and $|\K(R_{n}^{c})|$, and one sees that in the statement of Theorem \ref{thm:main} the orders of all of the groups involved are correct.

\subsection{Order arguments}\ \\

Let $M$ denote either the adjacency or Laplacian matrix of a graph.  Both the domain and codomain of the map defined by $M$ are equal to the free abelian group on the vertex set of the graph.  We visualize the elements of this group as \textit{configurations} on our graph.  Formally, for us a configuration on a graph $\Gamma$ will be a function $c \colon V(\Gamma) \to \ZZZ$.   We suggest that the reader imagine a picture of the graph with integers labeling the vertices.  Addition of configurations becomes vertex-wise addition of the labels.  For the rook's graph, such a visualization is quite easy; simply take an $n \times n$ grid and fill the squares with integers.

Now the cokernel of $M$ is still represented by the set of all configurations, but up to a certain equivalence.  For the Laplacian matrix, this equivalence is given by ``chip-firing'', which we will explain in a moment. First, observe that the set of configurations whose labels sum to zero form the smallest direct summand of the codomain of $L$ that contains the image of $L$. Therefore, the elements of the critical group can always be represented by configurations whose labels sum to zero. (In the Smith group, all configurations are permitted, not just those with labels summing to zero.)

Two configurations represent the same element of $\K(\Gamma)$ if one can be transformed into the other by a finite sequence of operations. The following are the operations we allow:
\begin{enumerate}
\item \textit{fire at a vertex}:  results in the label of that vertex decreasing by its degree, and the labels of its neighbors each increase by 1,
\item \textit{pull at a vertex}:  results in the label of that vertex increasing by its degree, and the labels of its neighbors each decrease by 1.
\end{enumerate}
We will sometimes refer to the labels in the grid as the number of chips on each vertex.  With a little thought the reader should be able to recognize that two configurations are equivalent precisely when their difference is in the image of the Laplacian.  Thus, we are indeed describing the cokernel of $L$. We give an example in Figure~\ref{fig:Equiv} of two equivalent configurations in the critical group of $R_4$. Here and throughout, we use empty space to denote zeros.

\begin{figure}[h]
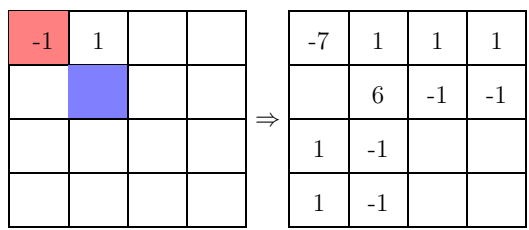

$$
{\renewcommand{\arraystretch}{1.8}
\scalebox{.8}{
\begin{tabular}{|c|c|c|c| }
 \hline	\cellcolor{red!50} -1 & 1 & \0 & \0 \\
 \hline \0 & \cellcolor{blue!50} & \0 & \0\\
 \hline \0 & \0 & \0 & \0\\
 \hline \0 & \0 & \0 & \0\\
 \hline  
\end{tabular}
$\Rightarrow$
\begin{tabular}{|c|c|c|c| }
 \hline	\ -7 \ & 1 & 1 & 1 \\
 \hline \0 & 6 & -1 & -1\\
 \hline 1 & \ -1 \ & \0 & \0\\
 \hline 1 & -1 & \0 & \0\\
 \hline  
\end{tabular}
}}
$$
\caption{Equivalent configurations in $\K(R_4)$, seen by firing at the red vertex and pulling at the blue.}\label{fig:Equiv}
\end{figure}

For the adjacency matrix, we can still play a similar game.  The elements of the Smith group can be viewed as configurations on the graph, up to a similar equivalence by a finite sequence of fires and pulls.  The only difference is that in the Smith group, it is only the neighbors' labels that change, and no change is made at the vertex that was fired or pulled. In Figure~\ref{fig:EquivS}, we give an example of two equivalent configurations in the Smith group of $R_4$.

\begin{figure}[h]
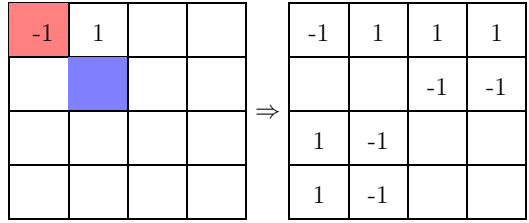

$$
{\renewcommand{\arraystretch}{1.8}
\scalebox{.8}{
\begin{tabular}{|c|c|c|c| }
 \hline	\cellcolor{red!50} -1 & 1 & \0 & \0 \\
 \hline \0 & \cellcolor{blue!50} & \0 & \0\\
 \hline \0 & \0 & \0 & \0\\
 \hline \0 & \0 & \0 & \0\\
 \hline  
\end{tabular}
$\Rightarrow$
\begin{tabular}{|c|c|c|c| }
 \hline	\ -1 \ & 1 & 1 & 1 \\
 \hline \0 & \0 & -1 & -1\\
 \hline 1 & \ -1 \ & \0 & \0\\
 \hline 1 & -1 & \0 & \0\\
 \hline  
\end{tabular}
}}
$$
\caption{Equivalent configurations in $S(R_4)$, seen by firing at the red vertex and pulling at the blue.}\label{fig:EquivS}
\end{figure}
Interestingly, the Smith and critical groups for both the rook's graph and its complement share many of the same generators. Overall, we have five main generators. The cyclic summands in our isomorphisms will turn out to be generated by these configurations and their images under certain automorphisms of our graph.  If $c$ is a configuration on a graph $\Gamma$ and $\sigma$ is an automorphism of $\Gamma$, then we can define a new configuration $\sigma(c)$ by $(\sigma(c))(v) = c(\sigma^{-1}(v))$, for all $v \in V(\Gamma)$.  In other words, the label at vertex $v$ in the configuration $c$ becomes the label at vertex $\sigma(v)$ in the configuration $\sigma(c)$.  For the rook's graph, this is again easy to visualize.  The automorphisms of $R_{n}$ we use are row swaps, column swaps, and the automorphism $\rho \colon V(R_{n}) \to V(R_{n})$ that reflects across the main diagonal, swapping the $i$th row with the $i$th column.  All the above comments apply to $R_{n}^{c}$ as well, since $R_{n}$ and $R_{n}^{c}$ share the same vertex set and have the same automorphisms.

To show that a particular configuration $c$ has order dividing $k$, we will exhibit a sequence of fires and pulls that transform the configuration $k \cdot c$ into the all-zero configuration.  If such a firing sequence exists for the configuration $k \cdot c$, then clearly such a firing sequence exists for the configuration $k \cdot \sigma(c)$; just apply the automorphism $\sigma$ to all vertices in the first firing sequence.  The important point is that when we consider a conjugate pair of configurations $c$ and $\sigma(c)$, their orders will be equal, though this order will depend on which group we are viewing the elements in.

We now define our main generators $c_{1}, c_{2}, c_{3}, c_{4}, c_{5}$.  We remind the reader that our vertex set is the squares of an $n \times n$ grid; we indicate the value of a configuration $c$ on the square in the $i$th row and $j$th column by $c(i,j)$.

\newpage
\begin{align*}
c_{1}(i,j) &= \begin{cases}
	-1, & \mbox{if $(i,j) = (1,1)$}\\
	1, & \mbox{if $(i,j) = (1,2)$}\\
	0, & \mbox{otherwise.}
	\end{cases}\\
c_{2}(i,j) &= \begin{cases}
	-(n-1), & \mbox{if $(i,j) = (1,1)$}\\
	1, & \mbox{if $i = 1$ and $2 \leq j \leq n$}\\
	0, & \mbox{otherwise.}
	\end{cases}\\
c_{3}(i,j) &= \begin{cases}
	-1, & \mbox{if $(i,j) = (1,1)$ or $(2,2)$}\\
	1, & \mbox{if $(i,j) = (1,2)$ or $(2,1)$}\\
	0, & \mbox{otherwise.}
	\end{cases}\\
c_{4}(i,j) &= \begin{cases}
	1, & \mbox{if $(i,j) = (1,1)$}\\
	0, & \mbox{otherwise.}
	\end{cases}\\
c_{5}(i,j) &= \begin{cases}
	-n(n-1)(n-2), & \mbox{if $(i,j) = (1,1)$}\\
	n(n-2), & \mbox{if $j = 1$ and $2 \leq i \leq n$}\\
	0, & \mbox{otherwise.}
	\end{cases}
\end{align*}

We illustrate these five configurations in Figure \ref{fig:maingens}.

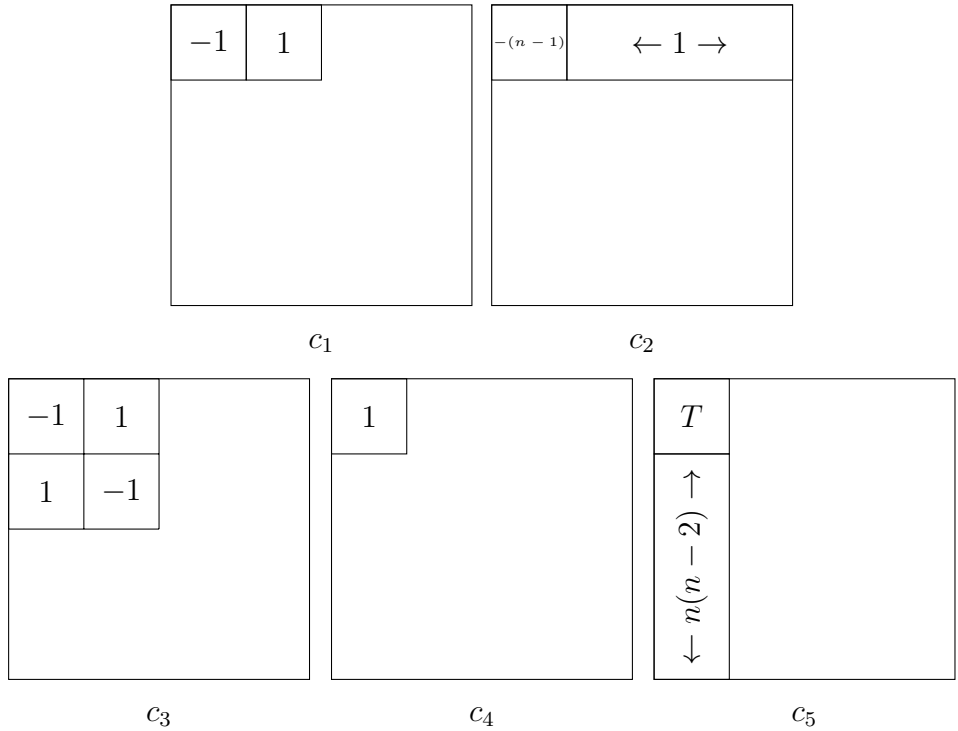
\begin{figure}[H]
$$
\begin{tikzpicture}
\draw[step=0.5cm,color=black] (0,0) rectangle (4,4);
\draw[step=0.5cm,color=black] (0,3) rectangle (1,4);
\draw[step=0.5cm,color=black] (1,3) rectangle (2,4);
\node at (.5,3.5){$-1$};
\node at (1.5,3.5){$1$};
\node at (2,-.5){$c_1$};
\end{tikzpicture}
\
\begin{tikzpicture}
\draw[step=0.5cm,color=black] (0,0) rectangle (4,4);
\draw[step=0.5cm,color=black] (0,3) rectangle (1,4);
\draw[step=0.5cm,color=black] (1,3) rectangle (4,4);
\node at (.5,3.5){{\fontsize{5}{7}$-(n-1)$}};
\node at (2.5,3.5){$\leftarrow 1 \rightarrow$};
\node at (2,-.5){$c_2$};
\end{tikzpicture}
$$
$$
\begin{tikzpicture}
\draw[step=0.5cm,color=black] (0,0) rectangle (4,4);
\draw[step=1cm,color=black] (0,2) grid (2,4);
\node at (.5,3.5){$-1$};
\node at (1.5,3.5){$1$};
\node at (.5,2.5){$1$};
\node at (1.5,2.5){$-1$};
\node at (2,-.5){$c_3$};
\end{tikzpicture}
\ \
\begin{tikzpicture}
\draw[step=0.5cm,color=black] (0,0) rectangle (4,4);
\draw[step=0.5cm,color=black] (0,3) rectangle (1,4);
\node at (.5,3.5){$1$};
\node at (2,-.5){$c_4$};
\end{tikzpicture}
\ \
\begin{tikzpicture}
\draw[step=0.5cm,color=black] (0,0) rectangle (4,4);
\draw[step=0.5cm,color=black] (0,0) rectangle (1,3);
\draw[step=0.5cm,color=black] (0,3) rectangle (1,4);
\node at (.5,3.5){$T$};
\node at (.5,1.5){\begin{sideways} $\leftarrow n(n-2)\rightarrow$ \end{sideways}};
\node at (2,-.5){$c_5$};
\end{tikzpicture}
$$
\caption{The main generators.  Here $T=-n(n-1)(n-2)$.}\label{fig:maingens}
\end{figure}
\newpage

\subsubsection{$\K(R_{n})$}\ \\
 We begin with the critical group of $R_n$, whose claimed decomposition is $$\K(R_n) \cong (\ZZ_{2n})^{(n-2)^2+1} \oplus  (\ZZ_{2n^2})^{2(n-2)}.$$ We will use $c_1, c_2,$ and $c_3$ to form the set of generators of $\K(R_n)$. We will begin by showing the order of $c_1$ must divide $2n^2$, and that the order of $c_2$ and $c_3$ must divide $2n$. Finally, we explain how to get the entire set of generators from the images of these three configurations under certain automorphisms of $R_{n}$.

To prove the orders of each configuration must divide the desired order, we simply give the firing sequence that gets the desired multiple to the all-zero configuration. We will denote the firing sequence of a configuration $c$ by $F(c)$.  Since it does not matter in which order the vertices are fired/pulled, we can exhibit the firing sequence as a function $F(c) \colon V(\Gamma) \to \ZZZ$ where the value on vertex $v$ tells you how many times to fire it if positive, and how many times to pull it if negative.   Again, we will usually visualize the firing sequence as labels on an $n \times n$ grid, though it is important to not confuse a configuration and a firing sequence.  In \cite{biggs}, $F(c)$ is called the representative vector of the firing sequence. For us, this will always be a vector such that $c - LF(c) = \vec{0}$, where $\vec{0}$ is the all-zero configuration.

We will start by formally defining the firing sequence for one configuration, $2n^2 \cdot c_1$, and then explain how to check a vertex's end value after a firing sequence. After doing so, we will simply display each firing sequence as a diagram on the $n \times n$ grid, as described above. We have the following.
$$F(2n^2\cdot c_1) = \begin{cases}
-(n+1) & \text{at vertex $(1,1)$,} \\
n+1 & \text{at vertex $(1,2)$,} \\
-1& \text{at vertices $(i,1)$ for $2 \leq i \leq n$,} \\
1& \text{at vertices $(i,2)$ for $2 \leq i \leq n$,}\\
0 & \text{elsewhere.}
\end{cases}$$

To show this firing sequence does take $2n^2 \cdot c_1$ to the all-zero configuration, we can check each vertex. First we check $(1,1)$. The initial value of this vertex is $-2n^2$, and then it gains $(n+1)(2n-2)$ by being pulled $n+1$ times. It gains $n+1$ chips from $(1,2)$, and loses $n-1$ chips from every vertex below it being pulled once. Thus, the final value on $(1,1)$ is \begin{align*}
 -2n^2 + (n+1)(2n-2) + n+1 - (n-1) &= -2n^2 + 2n^2 - 2n + 2n - 2 + 2 \\
 &= 0.
\end{align*}

Through the same reasoning (but reversing the signs), we find that the value at $(1,2)$ goes to $0$ as well. It is clear that the value at any other vertex in the top row ends at $0$, since it starts at $0$, gains $n+1$ from the firings at $(1,2)$, then loses $n+1$ from the pulls at $(1,1)$. As you can see, we don't actually need to check \textit{every} vertex individually; the computation is greatly reduced by noticing that certain groups of vertices are affected in the same way by the firing sequence. For example, every vertex in the leftmost column (except $(1,1)$) will have the same initial value, and take the same steps to determine its final value. Consider the vertex $(i,1)$, $i > 1$, with initial value $0$. It will gain $2n-2$ when it is pulled once, and lose $n+1$ chips from $(1,1)$.  It will lose $n-2$ from its other neighbors in that same column being pulled. Finally, it will gain one chip from $(i,2)$ being fired. Therefore, the final value on $(i,1)$ is \begin{align*}
 0 + (2n-2) - (n+1) - (n-2) + 1 &= 2n - 2 - n - 1 - n + 2 + 1 \\
 &= 0.
\end{align*}

We can again do the same thing for any vertex $(i,2)$, $i > 1$,  by simply reversing the signs in our previous argument. Finally, any vertex $(i,j)$ with $i >1$ and $j >2$ will gain one chip from $(i,2)$ and lose one from $(i,1)$, resulting in no change from $0$. Therefore, this firing sequence takes $2n^2 \cdot c_1$ to the all-zero configuration, i.e., $2n^{2} \cdot c_{1}$ is equivalent to the identity.


From now on, we will simply display the firing sequences as diagrams on the $n \times n$ grid and allow the reader to check that they work.

\begin{figure}[H]
\begin{center}
\begin{tikzpicture}
\draw[step=0.5cm,color=black] (0,0) rectangle (4,4);
\draw[step=0.5cm,color=black] (0,3) rectangle (1,4);
\draw[step=0.5cm,color=black] (1,3) rectangle (2,4);
\draw[step=0.5cm,color=black] (0,0) rectangle (1,3);
\draw[step=0.5cm,color=black] (1,0) rectangle (2,3);
\node at (.43,3.5){{\fontsize{5}{5} $-(n+1)$}};
\node at (1.45,3.5){{\fontsize{7}{8} $(n+1)$}};
\node at (.5,1.5){$-1$};
\node at (.6,2){$\uparrow$};
\node at (.6,1){$\downarrow$};
\node at (1.5,1.5){$1$};
\node at (1.5,2){$\uparrow$};
\node at (1.5,1){$\downarrow$};
\end{tikzpicture}
\caption*{$F(2n^2\cdot c_1)$}
\end{center}
\end{figure}
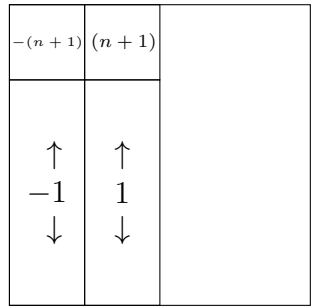
\begin{figure}[h]
\begin{center}
\begin{tikzpicture}
\draw[step=0.5cm,color=black] (0,0) rectangle (4,4);
\draw[step=0.5cm,color=black] (0,0) rectangle (1,3);
\draw[step=0.5cm,color=black] (1,3) rectangle (4,4);
\node at (0.5,3.5){$ -n $};
\node at (2.5,3.5){$\leftarrow 1 \rightarrow$};
\node at (.5,1.5){$-1$};
\node at (.6,2){$\uparrow$};
\node at (.6,1){$\downarrow$};
\end{tikzpicture}
\caption*{$F(2n \cdot c_2)$}
\end{center}
\end{figure}
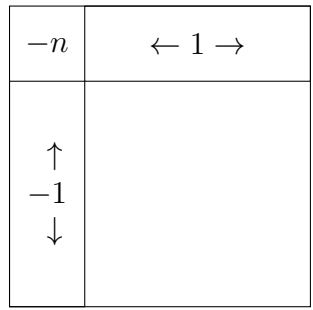
\begin{figure}[H]
\begin{center}
\begin{tikzpicture}
\draw[step=0.5cm,color=black] (0,0) rectangle (4,4);
\draw[step=1cm,color=black] (0,2) grid (2,4);
\node at (.5,3.5){$-1$};
\node at (1.5,2.5){$-1$};
\node at (1.5,3.5){$1$};
\node at (.5,2.5){$1$};
\end{tikzpicture}
\caption*{$F(2n\cdot c_3)$}
\end{center}
\end{figure}

The reader may notice that $F(2n \cdot c_3) = c_3$. This is because $c_3$ is actually an eigenvector of the Laplacian of $R_n$, with eigenvalue $2n$. In fact, $c_3$ will be an eigenvector for the adjacency matrix of $R_n$ as well, and even for the Laplacian and adjacency matrix for $R_n^c$. Thus, $c_3$ will appear as a generator for every group, and its order will be the shifted eigenvalue. We will now demonstrate how we construct all generators of our group using these main generators and certain automorphisms.

The configuration $c_2$ will be a single generator of order $2n$. The other $(n-2)^2$ generators of order $2n$ will come from $c_3$, through row and column swaps. We allow any permutation of the rows that fixes the first and last row, followed by any permutation of the columns that fixes the first and last column.  We can imagine sliding the innermost $-1$ of $c_{3}$ around the inside $(n-2) \times (n-2)$ square, with the ones on the outer row and column following it.

Finally, the $2(n-2)$ generators of order $2n^2$ will be the images of $c_1$ under the automorphisms we now list.  We allow any permutation of the columns that fixes the first and last column.  We also allow the reflection $\rho$, followed by any permutation of the rows that fixes the first and last row. 

For illustration, we list all generators of $\K(R_4)$ in Figure~\ref{fig:genKR4}.
 \begin{figure}[H]
$$
{\renewcommand{\arraystretch}{1.8}
\scalebox{.5}{
\begin{tabular}{|c|c|c|c|}
 \hline	-1 & 1 & \0 & \0 \\
 \hline 1 & -1 & \0 & \0\\
 \hline \0 & \0 & \0 & \0\\
 \hline \0 & \0 & \0 & \0\\
 \hline  
\end{tabular}
\
\begin{tabular}{|c|c|c|c|}
 \hline	-1 & 1 & \0 & \0 \\
 \hline \0 & \0 & \0 & \0\\
 \hline 1 & -1 & \0 & \0\\
 \hline \0 & \0 & \0 & \0\\
 \hline  
\end{tabular}
\
\begin{tabular}{|c|c|c|c|}
 \hline	-1 & \0 & 1 & \0 \\
 \hline \0 & \0 & \0 & \0\\
 \hline 1 & \0 & -1 & \0\\
 \hline \0 & \0 & \0 & \0\\
 \hline  
\end{tabular}
\
\begin{tabular}{|c|c|c|c|}
 \hline	-1 & \0 & 1 & \0 \\
 \hline 1 & \0 & -1 & \0\\
 \hline \0 & \0 & \0 & \0\\
 \hline \0 & \0 & \0 & \0\\
 \hline  
\end{tabular}
\
\begin{tabular}{|c|c|c|c|}
 \hline	-1 & 1 & \0 & \0 \\
 \hline \0 & \0 & \0 & \0\\
 \hline \0 & \0 & \0 & \0\\
 \hline \0 & \0 & \0 & \0\\
 \hline  
\end{tabular}
}}
$$
$$
{\renewcommand{\arraystretch}{1.8}
\scalebox{.5}{
\begin{tabular}{|c|c|c|c|}
 \hline	-1 & \0 & 1 & \0 \\
 \hline \0 & \0 & \0 & \0\\
 \hline \0 & \0 & \0 & \0\\
 \hline \0 & \0 & \0 & \0\\
 \hline  
\end{tabular}
\
\begin{tabular}{|c|c|c|c|}
 \hline	-1 & \0 & \0 & \0 \\
 \hline 1 & \0 & \0 & \0\\
 \hline \0 & \0 & \0 & \0\\
 \hline \0 & \0 & \0 & \0\\
 \hline  
\end{tabular}
\
\begin{tabular}{|c|c|c|c|}
 \hline	-1 & \0 & \0 & \0 \\
 \hline \0 & \0 & \0 & \0\\
 \hline 1 & \0 & \0 & \0\\
 \hline \0 & \0 & \0 & \0\\
 \hline  
\end{tabular}
\
\begin{tabular}{|c|c|c|c|}
 \hline	-3 & 1 & 1 & 1 \\
 \hline \0 & \0 & \0 & \0\\
 \hline \0 & \0 & \0 & \0\\
 \hline \0 & \0 & \0 & \0\\
 \hline  
\end{tabular}
}}
$$
\caption{Generators of $\K(R_4)$.}\label{fig:genKR4}
\end{figure}

\subsubsection{$S(R_{n})$}\ \\
We will now look at the Smith group, whose claimed decomposition is $$S(R_n) \cong (\ZZ_2)^{(n-2)^2} \oplus (\ZZ_{2(n-2)})^{2n-3} \oplus \ZZ_{2(n-1)(n-2)}.$$ For our generators of order $2$, we will use the images of $c_3$ under the same automorphisms as before, giving us $(n-2)^2$ total. For our generators of order $2(n-2)$, we start with $c_1$. The automorphisms we apply to $c_{1}$ here are the same as before, except we no longer require that the column permutations fix the last column (that is, we allow the $1$ to take the last spot in the top row) giving us a total of $2n-3$ such generators. Our sole generator of order $2(n-2)(n-1)$ will be $c_4$.  See Figure~\ref{fig:genSR4} for the list of generators of $S(R_4)$. 

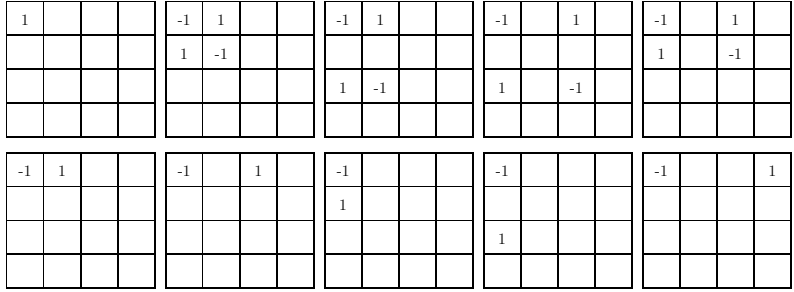
\begin{figure}[H]
$$
{\renewcommand{\arraystretch}{1.8}
\scalebox{.5}{
\begin{tabular}{|c|c|c|c|}
 \hline	1 & \0 & \0 & \0 \\
 \hline \0 & \0 & \0 & \0\\
 \hline \0 & \0 & \0 & \0\\
 \hline \0 & \0 & \0 & \0\\
 \hline  
\end{tabular}
\
\begin{tabular}{|c|c|c|c|}
 \hline	-1 & 1 & \0 & \0 \\
 \hline 1 & -1 & \0 & \0\\
 \hline \0 & \0 & \0 & \0\\
 \hline \0 & \0 & \0 & \0\\
 \hline  
\end{tabular}
\
\begin{tabular}{|c|c|c|c|}
 \hline	-1 & 1 & \0 & \0 \\
 \hline \0 & \0 & \0 & \0\\
 \hline 1 & -1 & \0 & \0\\
 \hline \0 & \0 & \0 & \0\\
 \hline  
\end{tabular}
\
\begin{tabular}{|c|c|c|c|}
 \hline	-1 & \0 & 1 & \0 \\
 \hline \0 & \0 & \0 & \0\\
 \hline 1 & \0 & -1 & \0\\
 \hline \0 & \0 & \0 & \0\\
 \hline  
\end{tabular}
\
\begin{tabular}{|c|c|c|c|}
 \hline	-1 & \0 & 1 & \0 \\
 \hline 1 & \0 & -1 & \0\\
 \hline \0 & \0 & \0 & \0\\
 \hline \0 & \0 & \0 & \0\\
 \hline  
\end{tabular}
}}
$$
$$
{\renewcommand{\arraystretch}{1.8}
\scalebox{.5}{
\begin{tabular}{|c|c|c|c|}
 \hline	-1 & 1 & \0 & \0 \\
 \hline \0 & \0 & \0 & \0\\
 \hline \0 & \0 & \0 & \0\\
 \hline \0 & \0 & \0 & \0\\
 \hline  
\end{tabular}
\
\begin{tabular}{|c|c|c|c|}
 \hline	-1 & \0 & 1 & \0 \\
 \hline \0 & \0 & \0 & \0\\
 \hline \0 & \0 & \0 & \0\\
 \hline \0 & \0 & \0 & \0\\
 \hline  
\end{tabular}
\
\begin{tabular}{|c|c|c|c|}
 \hline	-1 & \0 & \0 & \0 \\
 \hline 1 & \0 & \0 & \0\\
 \hline \0 & \0 & \0 & \0\\
 \hline \0 & \0 & \0 & \0\\
 \hline  
\end{tabular}
\
\begin{tabular}{|c|c|c|c|}
 \hline	-1 & \0 & \0 & \0 \\
 \hline \0 & \0 & \0 & \0\\
 \hline 1 & \0 & \0 & \0\\
 \hline \0 & \0 & \0 & \0\\
 \hline  
\end{tabular}
\
\begin{tabular}{|c|c|c|c|}
 \hline	-1 & \0 & \0 & 1 \\
 \hline \0 & \0 & \0 & \0\\
 \hline \0 & \0 & \0 & \0\\
 \hline \0 & \0 & \0 & \0\\
 \hline  
\end{tabular}
}}
$$
\caption{Generators of $S(R_n)$.}\label{fig:genSR4}
\end{figure}

We give the firing sequences proving that the order of these configurations must divide the desired orders below. We remind the reader that since we are working in the Smith group, the vertex that is fired or pulled is unchanged. Let P $= n^2-5n+5$.

\begin{figure}[H]
\begin{center}
\begin{tikzpicture}
\draw[step=0.5cm,color=black] (0,0) rectangle (4,4);
\draw[step=1cm,color=black] (0,2) grid (2,4);
\node at (.5,3.5){$-1$};
\node at (1.5,2.5){$-1$};
\node at (1.5,3.5){$1$};
\node at (.5,2.5){$1$};
\end{tikzpicture}
\caption*{$F(2\cdot c_3)$}
\end{center}
\end{figure}
\begin{figure}[H]
\begin{center}
\begin{tikzpicture}
\draw[step=0.5cm,color=black] (0,0) rectangle (4,4);
\draw[step=0.5cm,color=black] (0,3) rectangle (1,4);
\draw[step=0.5cm,color=black] (1,3) rectangle (2,4);
\draw[step=0.5cm,color=black] (0,0) rectangle (1,3);
\draw[step=0.5cm,color=black] (1,0) rectangle (2,3);
\node at (.43,3.5){{\fontsize{5}{5} $-(n-3)$}};
\node at (1.45,3.5){{\fontsize{8}{10} $n-3$}};
\node at (.5,1.5){$1$};
\node at (.5,2){$\uparrow$};
\node at (.5,1){$\downarrow$};
\node at (1.5,1.5){$-1$};
\node at (1.6,2){$\uparrow$};
\node at (1.6,1){$\downarrow$};
\end{tikzpicture}
\caption*{$F(2(n-2)\cdot c_1)$}
\end{center}
\end{figure}
\begin{figure}[H]
\begin{center}
\begin{tikzpicture}
\draw[step=0.5cm,color=black] (0,0) rectangle (4,4);
\draw[step=0.5cm,color=black] (0,3) rectangle (4,4);
\draw[step=0.5cm,color=black] (0,3) rectangle (1,4);
\draw[step=0.5cm,color=black] (0,0) rectangle (1,3);
\node at (.5,3.5){$P$};
\node at (2.5,3.5){$\leftarrow -(n-2) \rightarrow$};
\node at (.5,1.5){\begin{sideways}$\leftarrow -(n-2) \rightarrow$ \end{sideways}};
\node at (2.5,1.5){{\Large 1}};
\end{tikzpicture}
\caption*{$F(2(n-2)(n-1)\cdot c_4)$}
\end{center}
\end{figure}

\subsubsection{$\K(R_{n}^{c})$}\label{sec:KRC}\ \\
Now we will look at the critical and Smith group of the complement of the rook's graph. We'll begin with the critical group, whose claimed decomposition is $$\K(R_{n}^{c}) \cong \left(\ZZ_{n(n-2)}\right)^{(n-2)^{2}-1} \oplus \left(\ZZ_{n(n-1)(n-2)}\right)^{2} \oplus \left(\ZZ_{n^{2}(n-1)(n-2)}\right)^{2(n-2)}.$$  However, we find a much nicer proof by splitting up one component of size $n(n-1)(n-2)$ into one component of size $n(n-2)$ and one component of size $n-1$. We can do this because $n-1$ is relatively prime to both $n$ and $n-2$.   Thus we will show
\[
\K(R_{n}^{c}) \cong \left(\ZZ_{n(n-2)}\right)^{(n-2)^{2}} \oplus \ZZ_{n-1} \oplus \left(\ZZ_{n(n-1)(n-2)}\right) \oplus \left(\ZZ_{n^{2}(n-1)(n-2)}\right)^{2(n-2)}.
\]

Our main generator of order $n(n-2)$ will be $c_3$; the single generator of order $n-1$ will be $c_5$; the single generator of $n(n-1)(n-2)$ will be $c_2$; and the main generator of order $n^2(n-1)(n-2)$ will be $c_1$.

The automorphisms we apply to $c_{3}$ are the same as in the last two cases to create all $(n-2)^2$ generators. Similarly, the automorphisms applied to $c_1$ are the same as for $\K(R_n)$; this allows the $1$ to be in any square in the top row that is not the last one, and any square in the leftmost column that is not the last. The firing sequences establishing the divisibility of the orders of the configurations by our desired orders are given below. Let $Q = n^2 - n - 1$ and $U = n(n-2)$.

\begin{figure}[H]
\begin{center}
\begin{tikzpicture}
\draw[step=0.5cm,color=black] (0,0) rectangle (4,4);
\draw[step=0.5cm,color=black] (0,3) rectangle (1,4);
\draw[step=0.5cm,color=black] (1,3) rectangle (2,4);
\draw[step=0.5cm,color=black] (0,0) rectangle (1,3);
\draw[step=0.5cm,color=black] (1,0) rectangle (2,3);
\node at (.5,3.5){$-Q$};
\node at (1.5,3.5){$Q$};
\node at (1.5,1.5){$-1$};
\node at (1.6,2){$\uparrow$};
\node at (1.6,1){$\downarrow$};
\node at (.5,1.5){$1$};
\node at (.5,2){$\uparrow$};
\node at (.5,1){$\downarrow$};
\end{tikzpicture}
\caption*{$F(n^2(n-2)(n-1)\cdot c_1)$}
\end{center}
\end{figure}
\begin{figure}[H]
\begin{center}
\begin{tikzpicture}
\draw[step=0.5cm,color=black] (0,0) rectangle (4,4);
\draw[step=0.5cm,color=black] (0,0) rectangle (1,3);
\draw[step=0.5cm,color=black] (1,3) rectangle (4,4);
\node at (0.45,3.5){$-U$};
\node at (2.5,3.5){$\leftarrow (n-1) \rightarrow$};
\node at (.5,1.5){$1$};
\node at (.5,2){$\uparrow$};
\node at (.5,1){$\downarrow$};
\end{tikzpicture}
\caption*{$F(n(n-2)(n-1) \cdot c_2)$}
\end{center}
\end{figure}
\begin{figure}[H]
\begin{center}
\begin{tikzpicture}
\draw[step=0.5cm,color=black] (0,0) rectangle (4,4);
\draw[step=1cm,color=black] (0,2) grid (2,4);
\node at (.5,3.5){$-1$};
\node at (1.5,2.5){$-1$};
\node at (1.5,3.5){$1$};
\node at (.5,2.5){$1$};
\end{tikzpicture}
\caption*{$F(n(n-2) \cdot c_3)$}
\end{center}
\end{figure}

Again, the firing sequence for $c_{3}$ is itself.  The configuration $c_5$ is a somewhat special case. If $\rho$ is the reflection about the main diagonal, then $c_5 = n(n-2) \cdot \rho(c_2)$. Since $c_2$ has order dividing $n(n-2)(n-1)$, then $\rho(c_2)$ must as well. Therefore, $c_5$ must have order dividing $n-1$. 

Again, we list the generators of $\K(R_4^c)$ in Figure~\ref{fig:genKRc4} as an example.

\begin{figure}[H]
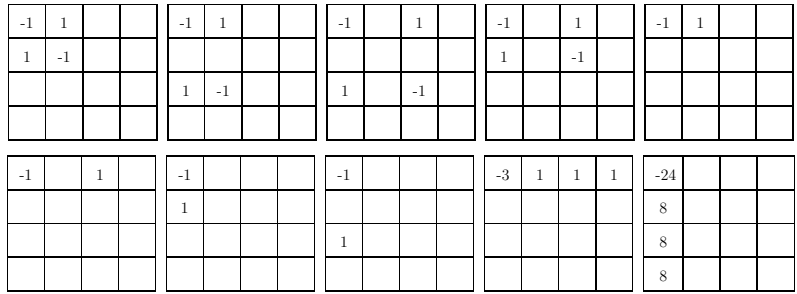

$$
{\renewcommand{\arraystretch}{1.8}
\scalebox{.5}{
\begin{tabular}{|c|c|c|c|}
 \hline	-1 & 1 & \0 & \0 \\
 \hline 1 & -1 & \0 & \0\\
 \hline \0 & \0 & \0 & \0\\
 \hline \0 & \0 & \0 & \0\\
 \hline  
\end{tabular}
\
\begin{tabular}{|c|c|c|c|}
 \hline	-1 & 1 & \0 & \0 \\
 \hline \0 & \0 & \0 & \0\\
 \hline 1 & -1 & \0 & \0\\
 \hline \0 & \0 & \0 & \0\\
 \hline  
\end{tabular}
\
\begin{tabular}{|c|c|c|c|}
 \hline	-1 & \0 & 1 & \0 \\
 \hline \0 & \0 & \0 & \0\\
 \hline 1 & \0 & -1 & \0\\
 \hline \0 & \0 & \0 & \0\\
 \hline  
\end{tabular}
\
\begin{tabular}{|c|c|c|c|}
 \hline	-1 & \0 & 1 & \0 \\
 \hline 1 & \0 & -1 & \0\\
 \hline \0 & \0 & \0 & \0\\
 \hline \0 & \0 & \0 & \0\\
 \hline  
\end{tabular}
\
\begin{tabular}{|c|c|c|c|}
 \hline	-1 & 1 & \0 & \0 \\
 \hline \0 & \0 & \0 & \0\\
 \hline \0 & \0 & \0 & \0\\
 \hline \0 & \0 & \0 & \0\\
 \hline  
\end{tabular}
}}
$$
$$
{\renewcommand{\arraystretch}{1.8}
\scalebox{.5}{
\begin{tabular}{|c|c|c|c|}
 \hline	-1 & \0 & 1 & \0 \\
 \hline \0 & \0 & \0 & \0\\
 \hline \0 & \0 & \0 & \0\\
 \hline \0 & \0 & \0 & \0\\
 \hline  
\end{tabular}
\
\begin{tabular}{|c|c|c|c|}
 \hline	-1 & \0 & \0 & \0 \\
 \hline 1 & \0 & \0 & \0\\
 \hline \0 & \0 & \0 & \0\\
 \hline \0 & \0 & \0 & \0\\
 \hline  
\end{tabular}
\
\begin{tabular}{|c|c|c|c|}
 \hline	-1 & \0 & \0 & \0 \\
 \hline \0 & \0 & \0 & \0\\
 \hline 1 & \0 & \0 & \0\\
 \hline \0 & \0 & \0 & \0\\
 \hline  
\end{tabular}
\
\begin{tabular}{|c|c|c|c|}
 \hline	-3 & 1 & 1 & 1 \\
 \hline \0 & \0 & \0 & \0\\
 \hline \0 & \0 & \0 & \0\\
 \hline \0 & \0 & \0 & \0\\
 \hline  
\end{tabular}
\
\begin{tabular}{|c|c|c|c|}
 \hline	\ -24 & \0 & \0 & \0 \\
 \hline 8 & \0 & \0 & \0\\
 \hline 8 & \0 & \0 & \0\\
 \hline 8 & \0 & \0 & \0\\
 \hline  
\end{tabular}
}}
$$
\caption{Generators of $\K(R_4^c)$.}\label{fig:genKRc4}
\end{figure}

\subsubsection{$S(R_{n}^{c})$}\ \\
The Smith group of $R_n^c$ has a very simple and nice decomposition: $$S(R_n^c) \cong (\ZZ_{n-1})^{2(n-1)} \oplus \ZZ_{(n-1)^2}.$$ We will use $c_1$ as our main generator of order $n-1$ and $c_4$ as our single generator of order $(n-1)^2$. We get $2(n-1)$ total generators from $c_1$ by moving the $1$ to be in any square in the top row or leftmost column. We'd like to note that our eigenvector $c_3$ still makes an appearance in this group, but its eigenvalue is $1$, so it's just the identity. This will be useful during our spanning arguments in the next section.

The firing sequences proving that the orders of our configurations must divide our desired order are listed below. We now let $T = (n-2)^2$.

\begin{figure}[H]
\begin{center}
\begin{tikzpicture}
\draw[step=0.5cm,color=black] (0,0) rectangle (4,4);
\draw[step=0.5cm,color=black] (0,3) rectangle (1,4);
\draw[step=0.5cm,color=black] (1,3) rectangle (2,4);
\draw[step=0.5cm,color=black] (0,0) rectangle (1,3);
\draw[step=0.5cm,color=black] (1,0) rectangle (2,3);
\node at (.5,3.5){{\tiny $n-2$}};
\node at (1.5,3.5){{\fontsize{5}{7}$-(n-2)$}};
\node at (1.5,1.5){$1$};
\node at (1.5,2){$\uparrow$};
\node at (1.5,1){$\downarrow$};
\node at (.5,1.5){$-1$};
\node at (.6,2){$\uparrow$};
\node at (.6,1){$\downarrow$};
\end{tikzpicture}
\caption*{$F((n-1)\cdot c_1)$}
\end{center}
\end{figure}
\begin{figure}[H]
\begin{center}
\begin{tikzpicture}
\draw[step=0.5cm,color=black] (0,0) rectangle (4,4);
\draw[step=0.5cm,color=black] (0,3) rectangle (4,4);
\draw[step=0.5cm,color=black] (0,3) rectangle (1,4);
\draw[step=0.5cm,color=black] (0,0) rectangle (1,3);
\node at (.5,3.5){$-T$};
\node at (2.5,3.5){$\leftarrow n-2 \rightarrow$};
\node at (.5,1.5){\begin{sideways}$\leftarrow n-2 \rightarrow$ \end{sideways}};
\node at (2.5,1.5){{\Large -1}};
\end{tikzpicture}
\caption*{$F((n-1)^2\cdot c_4)$}
\end{center}
\end{figure}

Finally, we finish this section by giving all the generators of $S(R_4^c)$ in Figure~\ref{fig:genSRc4}.
\begin{figure}[H]
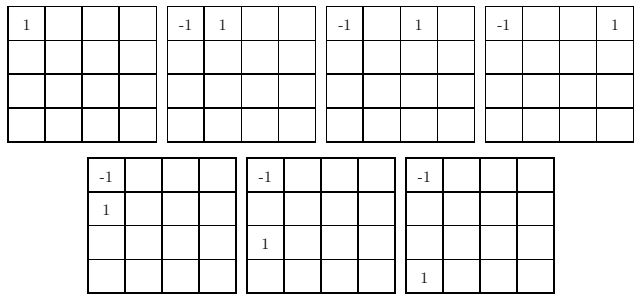

$$
{\renewcommand{\arraystretch}{1.8}
\scalebox{.5}{
\begin{tabular}{|c|c|c|c|}
 \hline	1 & \0 & \0 & \0 \\
 \hline \0 & \0 & \0 & \0\\
 \hline \0 & \0 & \0 & \0\\
 \hline \0 & \0 & \0 & \0\\
 \hline  
\end{tabular}
\
\begin{tabular}{|c|c|c|c|}
 \hline	-1 & 1 & \0 & \0 \\
 \hline \0 & \0 & \0 & \0\\
 \hline \0 & \0 & \0 & \0\\
 \hline \0 & \0 & \0 & \0\\
 \hline  
\end{tabular}
\
\begin{tabular}{|c|c|c|c|}
 \hline	-1 & \0 & 1 & \0 \\
 \hline \0 & \0 & \0 & \0\\
 \hline \0 & \0 & \0 & \0\\
 \hline \0 & \0 & \0 & \0\\
 \hline  
\end{tabular}
\
\begin{tabular}{|c|c|c|c|}
 \hline	-1 & \0 & \0 & 1 \\
 \hline \0 & \0 & \0 & \0\\
 \hline \0 & \0 & \0 & \0\\
 \hline \0 & \0 & \0 & \0\\
 \hline  
\end{tabular}
}}
$$
$$
{\renewcommand{\arraystretch}{1.8}
\scalebox{.5}{
\begin{tabular}{|c|c|c|c|}
 \hline	-1 & \0 & \0 & \0 \\
 \hline 1 & \0 & \0 & \0\\
 \hline \0 & \0 & \0 & \0\\
 \hline \0 & \0 & \0 & \0\\
 \hline  
\end{tabular}
\
\begin{tabular}{|c|c|c|c|}
 \hline	-1 & \0 & \0 & \0 \\
 \hline \0 & \0 & \0 & \0\\
 \hline 1 & \0 & \0 & \0\\
 \hline \0 & \0 & \0 & \0\\
 \hline  
\end{tabular}
\
\begin{tabular}{|c|c|c|c|}
 \hline	-1 & \0 & \0 & \0 \\
 \hline \0 & \0 & \0 & \0\\
 \hline \0 & \0 & \0 & \0\\
 \hline 1 & \0 & \0 & \0\\
 \hline  
\end{tabular}
}}
$$
\caption{Generators of $S(R_4^c)$.}\label{fig:genSRc4}
\end{figure}

The stark similarities between both the generating configurations and firing sequences for $\K(R_n), S(R_n), \K(R_n^c)$, and $S(R_n^c)$ lends to the idea that there may be some connection between each of these objects for a general graph. This connection evades the authors at this time.

\subsection{Spanning arguments}\ \\

Now we will show that our set of claimed generators for each group really do span all possible configurations.  We define the configuration
\[
q_{i,j} = \begin{cases}
-1, & \mbox{at vertex $(1,1)$}\\
1, & \mbox{at vertex  $(i,j)$}\\
0, & \mbox{elsewhere.}
\end{cases}
\]
As explained earlier, every element of the critical group of both $R_{n}$ and $R_{n}^{c}$ is represented by a configuration with vertex labels summing to zero.  Thus the set $$\{q_{i,j} \, \vert \, 1 \leq i, j \leq n \text{ with } (i,j) \neq (1,1) \}$$ is a generating set for both critical groups.  We will show that each of these $q_{i,j}$ can be expressed as a linear combination of the elements exhibited in the previous section, therefore showing that these elements do indeed form a generating set.

For the Smith groups, an obvious generating set consists of all configurations with a single vertex labeled $1$.  We will express all such configurations as linear combinations of the elements exhibited for the Smith groups in the last section.

\subsubsection{$\K(R_{n})$}\ \\
 We begin with the critical group of $R_{n}$. The configuration $c_1$ with the automorphisms previously mentioned are in fact $q_{1,j}$ and $q_{i,1}$ for $1 < i,j < n$. To get $q_{1,n}$, we just take $c_2$ and get rid of the unwanted ones using $-q_{1,j}$ for $1 < j < n$. For example, in $\K(R_4)$, we can use the following combination of generators.
$$
{\renewcommand{\arraystretch}{1.8}
\scalebox{.5}{
\begin{tabular}{|c|c|c|c| }
 \hline	1 & -1 & \0 & \0 \\
 \hline \0 & \0 & \0 & \0\\
 \hline \0 & \0 & \0 & \0\\
 \hline \0 & \0 & \0 & \0\\
 \hline  
\end{tabular}
+ 
\begin{tabular}{|c|c|c|c| }
 \hline	1 & \0 & -1 & \0 \\
 \hline \0 & \0 & \0 & \0\\
 \hline \0 & \0 & \0 & \0\\
 \hline \0 & \0 & \0 & \0\\
 \hline  
\end{tabular}
+
\begin{tabular}{|c|c|c|c| }
 \hline	-3 & 1 & 1 & 1 \\
 \hline \0 & \0 & \0 & \0\\
 \hline \0 & \0 & \0 & \0\\
 \hline \0 & \0 & \0 & \0\\
 \hline  
\end{tabular}
=
\begin{tabular}{|c|c|c|c|}
 \hline	-1 & \0 & \0 & 1 \\
 \hline \0 & \0 & \0 & \0\\
 \hline \0 & \0 & \0 & \0\\
 \hline \0 & \0 & \0 & \0\\
 \hline  
\end{tabular}
}}
$$
To get $q_{n,1}$, we need to start with the all-zero configuration and fire the top left once. Then, we get rid of the unwanted ones using $-c_2$ and $-q_{i,1}$ for $1 < i < n$. Going back to $R_4$, we do the following.

$$
{\renewcommand{\arraystretch}{1.8}
\scalebox{.5}{
\begin{tabular}{|c|c|c|c| }
 \hline	\ -6 \ & 1 & 1 & 1 \\
 \hline 1 & \0 & \0 & \0\\
 \hline 1 & \0 & \0 & \0\\
 \hline 1 & \0 & \0 & \0\\
 \hline  
\end{tabular}
+
\begin{tabular}{|c|c|c|c|}
 \hline	3 & -1 & -1 & -1 \\
 \hline \0 & \0 & \0 & \0\\
 \hline \0 & \0 & \0 & \0\\
 \hline \0 & \0 & \0 & \0\\
 \hline  
\end{tabular}
+
\begin{tabular}{|c|c|c|c|}
 \hline	1 & \0 & \0 & \0 \\
 \hline \0 & \0 & \0 & \0\\
 \hline -1 & \0 & \0 & \0\\
 \hline \0 & \0 & \0 & \0\\
 \hline  
\end{tabular}
+
\begin{tabular}{|c|c|c|c|}
 \hline	1 & \0 & \0 & \0 \\
 \hline -1 & \0 & \0 & \0\\
 \hline \0 & \0 & \0 & \0\\
 \hline \0 & \0 & \0 & \0\\
 \hline  
\end{tabular}
=
\begin{tabular}{|c|c|c|c|}
 \hline	-1 & \0 & \0 & \0 \\
 \hline \0 & \0 & \0 & \0\\
 \hline \0 & \0 & \0 & \0\\
 \hline 1 & \0 & \0 & \0\\
 \hline  
\end{tabular}
}}
$$
To get $q_{i,j}$ for $1<i,j<n$ we can just use the automorphism of $-c_3$ that puts a $1$ in spot $(i,j)$, and then get rid of the unwanted ones using $q_{1,j}$ and $q_{i,1}$. For example, we can do the following to get $q_{2,3}$ in $R_4$.
$$
{\renewcommand{\arraystretch}{1.8}
\scalebox{.5}{
\begin{tabular}{|c|c|c|c| }
 \hline	1 & \0 & -1 & \0 \\
 \hline -1 & \0 & 1 & \0\\
 \hline \0 & \0 & \0 & \0\\
 \hline \0 & \0 & \0 & \0\\
 \hline  
\end{tabular}
+
\begin{tabular}{|c|c|c|c| }
 \hline	-1 & \0 & 1 & \0 \\
 \hline \0 & \0 & \0 & \0\\
 \hline \0 & \0 & \0 & \0\\
 \hline \0 & \0 & \0 & \0\\
 \hline  
\end{tabular}
+ 
\begin{tabular}{|c|c|c|c| }
 \hline	-1 & \0 & \0 & \0 \\
 \hline 1 & \0 & \0 & \0\\
 \hline \0 & \0 & \0 & \0\\
 \hline \0 & \0 & \0 & \0\\
 \hline  
\end{tabular}
=
\begin{tabular}{|c|c|c|c|}
 \hline	-1 & \0 & \0 & \0 \\
 \hline \0 & \0 & 1 & \0\\
 \hline \0 & \0 & \0 & \0\\
 \hline \0 & \0 & \0 & \0\\
 \hline  
\end{tabular}
}}
$$
To get a $1$ in some spot in the bottom row, we can start with the all-zero configuration and fire the first spot of the column it's in and then get rid of all unwanted ones using the other $q_{i,j}$ we've built up. Analogously, if we want to get a $1$ in the rightmost column, we fire the spot in the same row in the leftmost column, and then get rid of unwanted ones.
$$
{\renewcommand{\arraystretch}{1.8}
\scalebox{.5}{
\begin{tabular}{|c|c|c|c| }
 \hline	1 & \ -6 \ & 1 & 1 \\
 \hline \0 & 1 & \0 & \0\\
 \hline \0 & 1 & \0 & \0\\
 \hline \0 & 1 & \0 & \0\\
 \hline  
\end{tabular}
+ 
\begin{tabular}{|c|c|c|c| }
 \hline	1 & \0 & -1 & \0 \\
 \hline \0 & \0 & \0 & \0\\
 \hline \0 & \0 & \0 & \0\\
 \hline \0 & \0 & \0 & \0\\
 \hline  
\end{tabular}
+
\begin{tabular}{|c|c|c|c| }
 \hline	1& \0 & \0 & -1 \\
 \hline \0 & \0 & \0 & \0\\
 \hline \0 & \0 & \0 & \0\\
 \hline \0 & \0 & \0 & \0\\
 \hline  
\end{tabular}
+
\begin{tabular}{|c|c|c|c|}
 \hline	1 & \0 & \0 & \0 \\
 \hline \0 & -1 & \0 & \0\\
 \hline \0 & \0 & \0 & \0\\
 \hline \0 & \0 & \0 & \0\\
 \hline  
\end{tabular}
+
\begin{tabular}{|c|c|c|c|}
 \hline	1 & \0 & \0 & \0 \\
 \hline \0 & \0 & \0 & \0\\
 \hline \0 & -1 & \0 & \0\\
 \hline \0 & \0 & \0 & \0\\
 \hline  
\end{tabular}
+
\begin{tabular}{|c|c|c|c|}
 \hline	-6 & 6 & \0 & \0 \\
 \hline \0 & \0 & \0 & \0\\
 \hline \0 & \0 & \0 & \0\\
 \hline \0 & \0 & \0 & \0\\
 \hline  
\end{tabular}
=
\begin{tabular}{|c|c|c|c|}
 \hline	-1 & \0 & \0 & \0 \\
 \hline \0 & \0 & \0 & \0\\
 \hline \0 & \0 & \0 & \0\\
 \hline \0 & 1 & \0 & \0\\
 \hline  
\end{tabular}
}}
$$

After obtaining $q_{i,n}$ or $q_{n,j}$ for $1 \leq i,j < n$, we can get $q_{n,n}$ using the same method.

\subsubsection{$S(R_{n})$}\ \\
For the Smith group $S(R_n)$, we must show a single $1$ can be put in any spot. The generators of $S(R_n)$ are almost exactly the same as $\K(R_n)$, so the spanning argument is almost identical. Since we are in the Smith group, however, we must worry about getting a $1$ in the top-left corner. But this comes easily from $c_4$. Instead of $c_2$, we have the automorphism of $c_1$ that puts a $1$ in spot $(1,n)$ and then we get rid of the $-1$ in the top left using $c_4$. Other than this, the argument is identical.

\subsubsection{$\K(R_{n}^{c})$}\ \\
For the critical group $\K(R_n^c)$, we again must show that we can get any $q_{i,j}$ for $(i,j) \neq (1,1)$. We can get $q_{1,j}$ for $1<j\leq n$ and $q_{i,1}$ for $1<i<n$ the same way we did for the critical group of the rook's graph; the fact that we are in the complement doesn't change anything since we never fire any vertices. Getting $q_{n,1}$ takes a little bit of work. We start with the all-zero configuration and then fire all of the vertices in the top row (except the first). Then we pull the top left $n-2$ times.
$$
{\renewcommand{\arraystretch}{1.8}
\scalebox{.5}{
\begin{tabular}{|c|c|c|c| }
 \hline	\0 & \cellcolor{red!50} & \cellcolor{red!50} & \cellcolor{red!50} \\
 \hline \0 & \0 & \0 & \0\\
 \hline \0 & \0 & \0 & \0\\
 \hline \0 & \0 & \0 & \0\\
 \hline  
\end{tabular}
$\Rightarrow$
\begin{tabular}{|c|c|c|c| }
 \hline	\cellcolor{blue!50} & \ -9 \ & \ -9 \ & \ -9 \ \\
 \hline \ 3 \ & 2 & 2 & 2\\
 \hline 3 & 2 & 2 & 2\\
 \hline 3 & 2 & 2 & 2\\
 \hline  
\end{tabular}
$\Rightarrow$
\begin{tabular}{|c|c|c|c| }
 \hline	\ 18 \ & -9 & -9 & -9 \\
 \hline 3 & \0 & \0 & \0\\
 \hline 3 & \0 & \0 & \0\\
 \hline 3 & \0 & \0 & \0\\
 \hline  
\end{tabular}
}}
$$
If we now add $(n-1)^2\cdot c_2$ to what we have, we obtain a configuration with $n-1$ in all of the spots of the leftmost column other than the top, which has $-(n-1)^2$.

$$
{\renewcommand{\arraystretch}{1.8}
\scalebox{.5}{
\begin{tabular}{|c|c|c|c| }
 \hline	\ -9 \ & \0 & \0 & \0 \\
 \hline 3 & \0 & \0 & \0\\
 \hline 3 & \0 & \0 & \0\\
 \hline 3 & \0 & \0 & \0\\
 \hline  
\end{tabular}
}}
$$
Notice that this configuration is similar to $c_5$ which has $n(n-2)$ in all of the spots of the leftmost column, except the top which has $-n(n-1)(n-2)$. Since $n-1$ and $n(n-2)$ are relatively prime, by the division algorithm, there must exist integers $a$ and $b$ such that $a(n-1)+ bn(n-2) = 1$. Letting $\rho$ be our reflection across the main diagonal again, this means that $\rho(c_2)$ is a linear combination of $c_5$ and the element above. With $\rho(c_2)$ in hand, we can get $q_{n,1}$ the same way we obtained $q_{1, n}$. Getting $q_{i,j}$ for $1<i,j<n$ can be done in much the same way as we did in the non-complement graph. We use the automorphism of $-c_3$ that puts a $1$ in spot $(i,j)$, and then get rid of the unwanted ones using $q_{1,j}$ and $q_{i,1}$.
Lastly, to get a $1$ in some spot in the bottom row not equal to $(n,1)$ or $(n,n)$, we can start with the all-zero configuration and fire the top left. Then we pull the first spot of the column it's in and get rid of all unwanted numbers using the other $q_{i,j}$ we've built up. This gives us $q_{n,j}$, $1 \leq j < n$, and we obtain $q_{i,n}$, $1 \leq i < n$, in an analogous way.
$$
{\renewcommand{\arraystretch}{1.8}
\scalebox{.5}{
\begin{tabular}{|c|c|c|c| }
 \hline	-9 & \0 & \0 & \0 \\
 \hline \0 & 1 & 1 & 1\\
 \hline \0 & 1 & 1 & 1\\
 \hline \0 & 1 & 1 & 1\\
 \hline  
\end{tabular}
+ 
\begin{tabular}{|c|c|c|c| }
 \hline	\0 & 9 & \0 & \0 \\
 \hline -1 & \0 & -1 & -1\\
 \hline -1 & \0 & -1 & -1\\
 \hline -1 & \0 & -1 & -1\\
 \hline  
\end{tabular}
+
\begin{tabular}{|c|c|c|c| }
 \hline	9 & -9 & \0 & \0 \\
 \hline \0 & \0 & \0 & \0\\
 \hline \0 & \0 & \0 & \0\\
 \hline \0 & \0 & \0 & \0\\
 \hline  
\end{tabular}
+
\begin{tabular}{|c|c|c|c|}
 \hline	\ -3 \ & \0 & \0 & \0 \\
 \hline 1 & \0 & \0 & \0\\
 \hline 1 & \0 & \0 & \0\\
 \hline 1 & \0 & \0 & \0\\
 \hline  
\end{tabular}
+
\begin{tabular}{|c|c|c|c|}
 \hline	1 & \0 & \0 & \0 \\
 \hline \0 & \0 & \0 & \0\\
 \hline \0 & -1 & \0 & \0\\
 \hline \0 & \0 & \0 & \0\\
 \hline  
\end{tabular}
+
\begin{tabular}{|c|c|c|c|}
 \hline	1 & \0 & \0 & \0 \\
 \hline \0 & -1 & \0 & \0\\
 \hline \0 & \0 & \0 & \0\\
 \hline \0 & \0 & \0 & \0\\
 \hline  
\end{tabular}
=
\begin{tabular}{|c|c|c|c|}
 \hline	-1 & \0 & \0 & \0 \\
 \hline \0 & \0 & \0 & \0\\
 \hline \0 & \0 & \0 & \0\\
 \hline \0 & 1 & \0 & \0\\
 \hline  
\end{tabular}
}}
$$
After obtaining $q_{i,n}$ or $q_{n,j}$ for $1 \leq i,j < n$, we again get $q_{n,n}$ using the same method.

\subsubsection{$S(R_{n}^{c})$}\ \\
Looking at $S(R_n^c)$, our spanning argument becomes very easy. The configuration $c_4$ and all automorphisms of $c_1$ allow us to get a single $1$ in any spot in the top row or left column. We remind the reader that $c_3$, our eigenvector, now has eigenvalue $1$, and thus $c_3$ is equal to the identity. It follows that any automorphism of $c_3$ will be the identity as well. So we can put a $1$ in any spot in the bottom right $(n-1) \times (n-1)$ square by applying some automorphism to $-c_3$ that fixes the top left spot, and then getting rid of the unwanted ones and negative ones in the top row and leftmost column.

\section{Acknowledgements}
This work was supported by the James Madison University's Tickle Fund.

\bibliographystyle{plain}
\bibliography{rook}

\end{document}